\documentclass[12pt]{amsart}
\usepackage{geometry}                % See geometry.pdf to learn the layout options. There are lots.
\usepackage{hyperref}
\usepackage{graphicx}
\usepackage{amssymb}
\usepackage{epstopdf}
\newtheorem{theorem}{Theorem}[section]

\newtheorem{conjecture}[theorem]{Conjecture}

\theoremstyle{definition}

\theoremstyle{plain}

\renewcommand{\P}{\mathbb{P}}

\begin{document}
\title{Determinants of Seidel matrices and a conjecture of Ghorbani}

\author[Douglas~Rizzolo]{ \ Douglas~Rizzolo}
\address{Department of Mathematical Sciences, University of Delaware, 
Newark, DE 19808}
\email{drizzolo@udel.edu}

\begin{abstract}
Let $G_n$ be a simple graph on $V_n=\{v_1,\dots, v_n\}$.  The Seidel matrix $S(G_n)$ of $G_n$ is the $n\times n$ matrix whose $(ij)$'th entry, for $i\neq j$ is $-1$ if $v_i\sim v_j$ and $1$ otherwise, and whose diagonal entries are $0$.  We show that the proportion of simple graphs $G_n$ such that $\det(S(G_n))\geq n-1$ tends to one as $n$ tends to infinity.
\end{abstract}



\maketitle

Let $G_n$ be a simple graph on $V_n=\{v_1,\dots, v_n\}$.  The Seidel matrix $S(G_n)$ of $G_n$ is the $n\times n$ matrix whose $(ij)$'th entry, for $i\neq j$ is $-1$ if $v_i\sim v_j$ and $1$ otherwise, and whose diagonal entries are $0$.  The Seidel matrix of $G_n$ is related to the adjacency matrix $A(G_n)$ of $G_n$ by $S(G_n)=J-I-2A(G_n)$, where $I$ is the identity matrix and $J$ is the all ones matrix.  Let $\lambda_1\leq\cdots \leq \lambda_n$ be the eigenvalues of $S(G_n)$.  The Seidel energy of $G_n$ , defined as
\[ \mathcal{S}(G_n) = \sum_{i=1}^n |\lambda_i|,\]
was introduced by Haemers in \cite{Haemers} as a notion of energy that was invariant under Seidel switching and taking graph complements.  In \cite{Haemers} it was conjectured that $\mathcal{S}(G_n) \geq 2n-2$.  This conjecture was confirmed for special families of graphs in \cite{Ghorbani} and \cite{MR3496590}, and was recently confirmed in full generality in \cite{Ak}.  In investigating Haemers' conjecture Ghorbani \cite{Ghorbani} showed that the following conditions are equivalent
\begin{enumerate}
\item $\det(S(G_n)) \geq n-1$,
\item $\displaystyle \sum_{i= 1}^n |\lambda_i|^p \geq (n-1)^p + n-1$ for all $0<p<2$,
\end{enumerate}
and made the following conjecture:

\begin{conjecture}
The proportion of graphs $G_n$ on $V_n$ satisfying $\det(S(G_n)) \geq n-1$ tends to $1$ as $n$ tends to infinity.
\end{conjecture}

In this short note we confirm this conjecture.  We remark that, in \cite{Ghorbani}, an infinite family of graphs $G_n$ with $\det(S(G_n))=0$ was given, so it is known that not all graphs satisfy the above conditions.  However, it was shown in \cite{Ak} that, for all $G_n$ and $0<p<2$,  
\[ \sum_{i= 1}^n |\lambda_i|^p > (n-1)^p + n-2.\]
This raises the question of finding a sharp lower bound for the quantities $\sum_{i= 1}^n |\lambda_i|^p$, which we leave open.

Observe that if $G_n$ is a uniformly random graph with vertex set $V_n$, then the entries of $S(G_n)$ above the diagonal are independent and, for $i<j$, 
\[ \P(S(G_n)_{ij}=-1) = \P(S(G_n)_{ij}=1) =\frac{1}{2} .\]
Moreover, the diagonal entries of $S(G_n)$, being constants and thus independent of everything, are independent of each other and of the entries of $S(G_n)$ above the diagonal.  Consequently, in the language of the random matrix literature, $S(G_n)$ is a \textit{real Wigner matrix} \cite{AGZ}.  Note however, that care must be taken when reading the random matrix literature because some authors require their Wigner matrices to satisfy additional assumptions on the diagonal matrices. 

The study of determinants of real Wigner matrices has a long history in the random matrix literature, but techniques have only recently advanced to the point where they can handle matrices whose entries above the diagonal can only take two values \cite{TaoVu}.  Using these results, we are able to give a short resolution of the conjecture. 

\begin{theorem}
Let $G_n$ be a uniformly random graph on vertex set $\{v_1,\dots, v_n\}$.  Then $\lim_{n\to\infty} \P(\det(S(G_n)) \geq n-1) =1$.
\end{theorem}

\begin{proof}
Let $\lambda_1\leq \lambda_2\leq \cdots \leq \lambda_n$ be the eigenvalues of $S(G_n)$, with multiplicities.  Fix $\epsilon>0$.  Wigner's semicircle law, see e.g.\ \cite[Theorem 2.1.1]{AGZ} implies that for every $b\geq 0$ and $\delta>0$,
\begin{equation}\label{eq wigner} \lim_{n\to\infty} \P\left( \left| \frac{\#\{i : |\lambda_i|\geq b\sqrt{n}\}}{n} - \frac{1}{\pi}\int_b^2 \sqrt{4-x^2} dx\right| > \delta\right)=0.\end{equation}
Since
\[ \frac{1}{\pi}\int_0^2 \sqrt{4-x^2} dx =1,\]
we can fix $b>0$ so that
\[   \frac{1}{\pi}\int_b^2 \sqrt{4-x^2} dx > \frac{1}{2}.\]
Define $c= \pi^{-1}\int_b^2 \sqrt{4-x^2} dx $.  Fixing $\delta$ such that 
\[0 < \delta < c - \frac{1}{2},\]
it follows from Equation \eqref{eq wigner} that there exists $N_1$ such that $n\geq N_1$ implies that 
\[\P\left( \frac{\#\{i : |\lambda_i|\geq b\sqrt{n}\}}{n} \geq  c-\delta \right) \geq 1-\epsilon.\]
In particular, with high probability the proportion of eigenvalues that are larger than $b\sqrt{n}$ is bounded above, and away from, $1/2$.  Using \cite[Theorem 1.2]{Versh}, we see that for some $C,r >0$ not depending on $\epsilon$,
\[ \P(\min_i |\lambda_i| \leq \epsilon n^{-1/2}) \leq C \epsilon^{1/9} + 2e^{-n^r},\]
for all $n$.  

On the event that 
\[ \frac{\#\{i : |\lambda_i|\geq b\sqrt{n}\}}{n} \geq  c-\delta \quad \textrm{and}\quad \min_i |\lambda_i| \geq \epsilon n^{-1/2},\]
we use the first condition to approximate the eigenvalues it applies to and the second for the eigenvalues smaller than $b\sqrt{n}$ in absolute value to see that
\[ |\det(S(G))| = \prod_{i=1}^n |\lambda_i|  \geq (b\sqrt{n})^{(c-\delta)n}(\epsilon n^{-1/2})^{(1-(c-\delta))n}=b^{(c-\delta)n}\epsilon^{(1-(c-\delta))n} n^{\frac{2(c-\delta)-1}{2}n} .\]
Since $(2(c-\delta)-1)/2 >0$, the last term in the product dominates, so there exists $N_2$ such that $n\geq N_2$ implies that $ |\det(S(G))| \geq n-1$.

Therefore, for $n\geq \max(N_1,N_2)$, we have that
\[\begin{split} \P( |\det(S(G))| \geq n-1) & \geq \P\left( \frac{\#\{i : |\lambda_i|\geq b\sqrt{n}\}}{n} \geq  c-\delta, \min_i |\lambda_i| \geq \epsilon n^{-1/2}\right)\\
& \geq 1- \epsilon - C \epsilon^{1/9} - 2e^{-n^r}
.\end{split}\]
From this it follows immediately that $\P( |\det(S(G))| \geq n-1)  \to 1$.
\end{proof}

The argument we've given in fact proves the much stronger result.

\begin{theorem}
For every $0\leq \alpha <1/2$ we have $\P( |\det(S(G))| \geq n^{\alpha n})  \to 1$.
\end{theorem}

It seems likely that the exact asymptotic growth of the determinant can be determined using very recent results from the random matrix literature \cite{BM18}, but note that Seidel matrices do not fit the definition of Wigner matrices in that paper and we have not pursued this direction.

\bibliographystyle{plain}
\bibliography{SeidelBib}

\end{document}